\documentclass[11pt]{amsart}

\usepackage{tikz}
\usepackage{graphicx} 
\usetikzlibrary{decorations.markings}
\usetikzlibrary{decorations.pathreplacing}

\usepackage{tikz}
\usepackage{dbnsymb} 

\usetikzlibrary{knots}
\usetikzlibrary{decorations.markings}

\usepackage{graphicx} 

\usepackage{t1enc}
\usepackage{amsthm,amssymb}
\usepackage{latexsym}
\usepackage{amsmath} 
\usepackage{graphics}
\usepackage{epsfig,multicol}
\usepackage{amsfonts}
\usepackage[latin1]{inputenc}
\usepackage[english]{babel}
\usepackage{ mathrsfs }

\newtheorem{theorem}{Theorem}
\newtheorem*{theorem*}{Theorem}
\newtheorem*{proposition*}{Proposition}
\newtheorem{proposition}{Proposition}
\newtheorem{lemma}{Lemma}
\newtheorem*{lemma*}{Lemma}

\newtheorem{definition}{Definition}

\newtheorem{remark}{Remark}

\newtheorem{example}{Example}

\def\hpic #1 #2 {\mbox{$\begin{array}[c]{l} \epsfig{file=#1,height=#2}
\end{array}$}}
 
\def\vpic #1 #2 {\mbox{$\begin{array}[c]{l} \epsfig{file=#1,width=#2}
\end{array}$}}

\newcommand  {\rmn}\romannumeral

 \newcommand{\CL}[0]{\mathcal{L}}

\begin{document}
\title{Positive oriented Thompson  links}
\author{Valeriano Aiello} 
\address{Valeriano Aiello,
Mathematisches Institut, Universit\"at Bern,  Alpeneggstrasse 22, 3012 Bern, Switzerland
}\email{valerianoaiello@gmail.com}
\author{Sebastian Baader} 
\address{Sebastian Baader,
Mathematisches Institut, Universit\"at Bern, Sidlerstrasse 5, 3012 Bern, Switzerland
}\email{sebastian.baader@unibe.ch}

\begin{abstract} 
We prove that the links associated with positive elements of the oriented subgroup of the Thompson group are positive.
\end{abstract}

\maketitle

\centerline{\it In memory of Vaughan F. R. Jones}\bigskip


\section{Introduction}

The Thompson group $F$ (along with its brothers $T$ and $V$) was introduced by R. Thompson in the sixties and has received a great deal of attention.
Indeed, several equivalent definitions appeared in the literature, for instance as a subgroup of piecewise linear homeomorphisms of $[0, 1]$, as a diagram group, as pairs of planar rooted binary trees, and as strand diagrams.
Motivated by the study of subfactors, Vaughan Jones started a new fascinating research program centred on the unitary representations of the Thompson groups.
In particular, Jones' recent work on the representation theory of Thompson's group~$F$ gave rise to a combinatorial model for links, where elements of~$F$ define links in a similar way as elements of the braid groups~\cite{Jo14}.
Unfortunately, the links arising from $F$ do not admit a natural orientation.
For this reason, Jones introduced the so-called oriented subgroup $\vec F$.
The links associated with the elements of $\vec{F}$ come with a natural orientation. 
The Thompson groups are as good knot constructors as the braid groups. In fact,   
all unoriented and oriented links can be produced by means of elements of $F$ and $\vec{F}$, respectively \cite{Jo14, A}.
The Thompson group $F$ contains a natural positive monoid $F_+$ generated by countably many generators $x_1$, $x_2$, $x_3$, \ldots
and its elements are called \emph{positive}. 
The oriented subgroup contains the submonoid $\vec{F}_+$ consisting of its positive elements. 
The purpose of this paper is to show that the notions of positivity for elements of~$\vec{F}$ and for links are compatible, in the following sense.

\begin{theorem}
For any $g\in \vec{F}_+$, the oriented link $\vec{\CL}(g)$ admits a positive  diagram.
\end{theorem}

Each positive element of the oriented group~$\vec{F}$ admits a description by means of a single finite 
rooted ternary trees $T$, which is turned into a binary tree 
via a   transform $\alpha(T)$ of the original tree.
The associated links 
 are defined diagrammatically. 
 This is illustrated in the first two figures, and defined in the next section. 
Positive oriented Thompson links are complicated in that their defining diagrams are highly non-minimal. In particular, there tend to be a lot of unknotted components. As we will see, removing these trivial components, together with another type of simplification, results in 
positive link diagrams whose crossing number is bounded above by the number of right leaves of the plane ternary tree~$T$. This is illustrated at the bottom of Figure \ref{knot5_2}, where the resulting link is the positive twist knot~$5_2$.

\begin{figure}[htb]
\begin{center}
\raisebox{-0mm}{\includegraphics[scale=0.6]{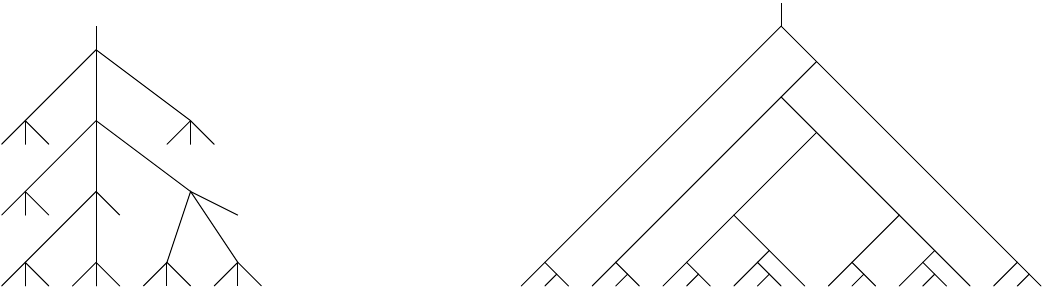}}
\caption{The $4$-regular rooted tree $T$ and its binary transform $\alpha(T)$.} \label{tree5_2}
\end{center}
\end{figure}

\begin{figure}[htb]
\begin{center}
\raisebox{-0mm}{\includegraphics[scale=0.5]{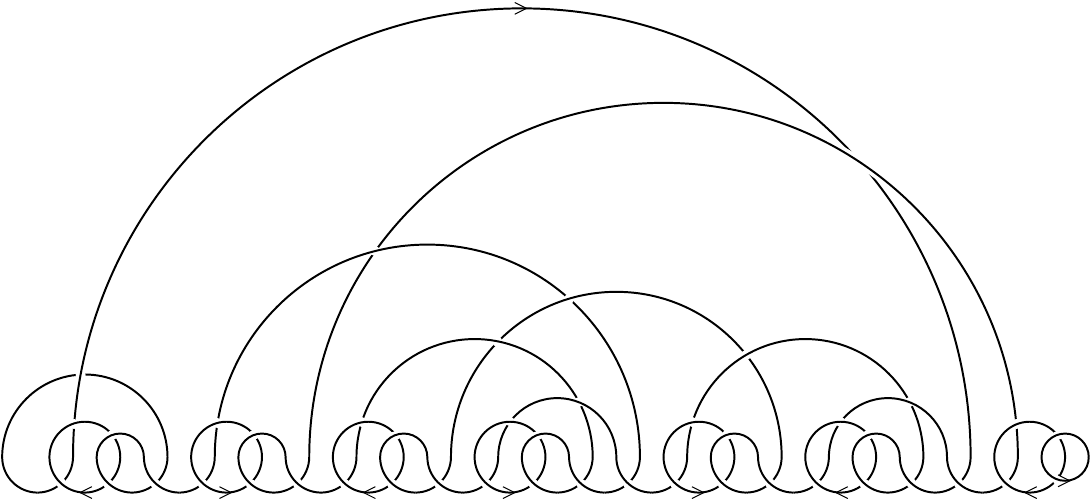}}

\bigskip
\bigskip
\raisebox{-0mm}{\includegraphics[scale=0.5]{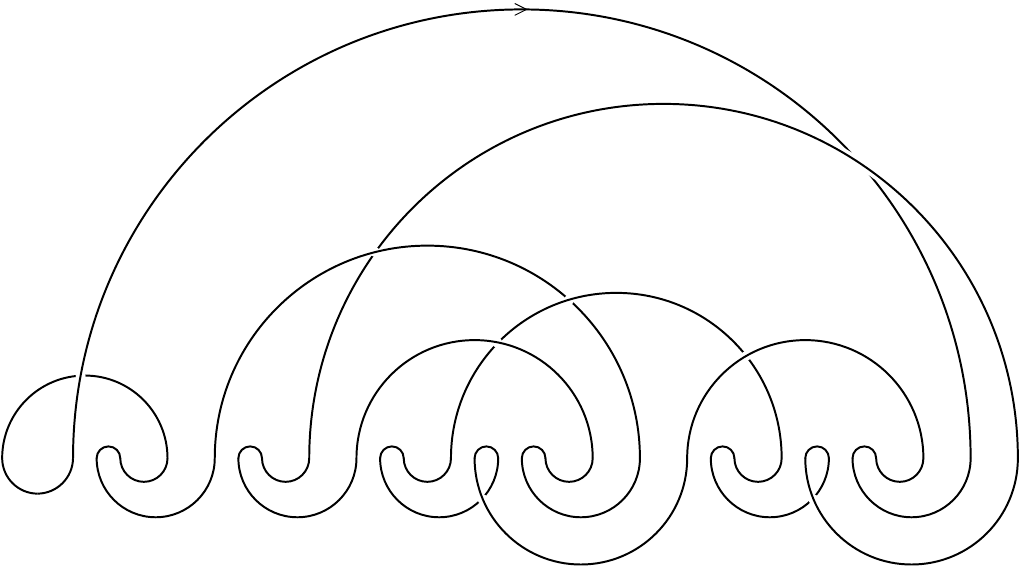}}

\caption{Positive oriented Thompson link associated with $T$, 
 and its knotted component $5_2$.} \label{knot5_2}
\end{center}
\end{figure}
 
\section{Preliminaries and notation} \label{secdue}

In this section we review the basic definitions and properties of the Thompson group $F$, the oriented Thompson group $\vec{F}$, the Brown-Thompson group $F_3$, and Jones's construction of knots from elements of Thompson group $F$.
The reader is referred to \cite{CFP, B} for more information
on the Thompson group,
to \cite{Brown} for $F_3$, 
to \cite{Jo14, GS, Ren} for $\vec{F}$,
to \cite{Jo14, Jo18, A} for the construction of links from elements of $F$ and $\vec{F}$.

The Thompson group   admits the following infinite presentation
$$
F=\langle x_0, x_1, \ldots \; | \; x_nx_k=x_kx_{n+1} \quad \forall \; k<n\rangle\; .
$$
Note that the elements $x_0$ and $x_1$ are enough to generate $F$.
The monoid generated by $x_0, x_1, \ldots$ is denoted by $F_+$ and its elements are said to be positive.
In this paper we will make use of a graphical description of the elements of $F$.
Every element of $F$ can be described by a pair of rooted planar binary trees $(T_+,T_-)$ with the same number of leaves  \cite{CFP}.
These pairs of binary trees are called binary tree diagrams.
We draw such tree diagrams in the plane, with one tree upside down on top of the other and the leaves sitting on the natural numbers of the $x$-axis (see Figure \ref{genF} for the generators of $F$).
The binary tree $T_+$, namely the one in the upper-half plane, is called the top tree. The other one is the so-called bottom tree. 
Two pairs of binary trees are equivalent when they differ by a pair of opposing carets, see Figure \ref{caret}.
This equivalence relation allows to define the multiplication in $F$ by the formula $(T_+,T)\cdot (T,T_-):=(T_+,T_-)$. The trivial element is represented by any pair $(T,T)$ and $(T_+,T_-)^{-1}=(T_-,T_+)$.

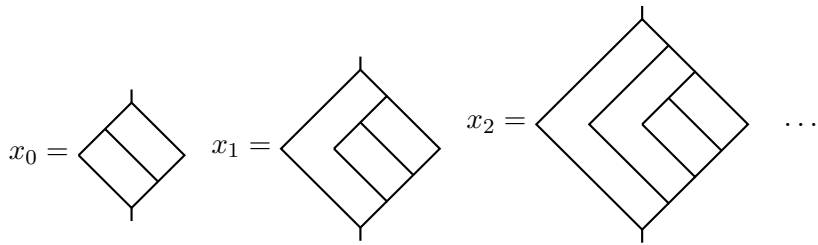
\begin{figure}[h]
\[
\begin{tikzpicture}[x=.35cm, y=.35cm,
    every edge/.style={
        draw,
      postaction={decorate,
                    decoration={markings}
                   }
        }
]

\node at (-1.5,0) {$\scalebox{1}{$x_0=$}$};
\node at (-1.25,-3) {\;};

\draw[thick] (0,0) -- (2,2)--(4,0)--(2,-2)--(0,0);
 \draw[thick] (1,1) -- (2,0)--(3,-1);

 \draw[thick] (2,2)--(2,2.5);

 \draw[thick] (2,-2)--(2,-2.5);

\end{tikzpicture}
\;\;
\begin{tikzpicture}[x=.35cm, y=.35cm,
    every edge/.style={
        draw,
      postaction={decorate,
                    decoration={markings}
                   }
        }
]

\node at (-3.5,0) {$\scalebox{1}{$x_1=$}$};
\node at (-1.25,-3.25) {\;};

\draw[thick] (2,2)--(1,3)--(-2,0)--(1,-3)--(2,-2);

\draw[thick] (0,0) -- (2,2)--(4,0)--(2,-2)--(0,0);
 \draw[thick] (1,1) -- (2,0)--(3,-1);

 \draw[thick] (1,3)--(1,3.5);
 \draw[thick] (1,-3)--(1,-3.5);

\end{tikzpicture}
\;\;
\begin{tikzpicture}[x=.35cm, y=.35cm,
    every edge/.style={
        draw,
      postaction={decorate,
                    decoration={markings}
                   }
        }
]

\node at (-5.5,0) {$\scalebox{1}{$x_2=$}$};
\node at (6,0) {$\ldots$};

\draw[thick] (2,2)--(1,3)--(-2,0)--(1,-3)--(2,-2);
\draw[thick] (1,3)--(0,4)--(-4,0)--(0,-4)--(1,-3); 

\draw[thick] (0,0) -- (2,2)--(4,0)--(2,-2)--(0,0);
 \draw[thick] (1,1) -- (2,0)--(3,-1);

 \draw[thick] (0,4)--(0,4.5);
 \draw[thick] (0,-4)--(0,-4.5);

\end{tikzpicture}
\]
\caption{The generators of $F$.}\label{genF}
\end{figure}

\begin{figure}[h] 
\[
\begin{tikzpicture}[x=1cm, y=1cm,
    every edge/.style={
        draw,
      postaction={decorate,
                    decoration={markings}
                   }
        }
]

\draw[thick] (0,0)--(.5,.5)--(1,0)--(.5,-.5)--(0,0);
\draw[thick] (.5,.5)--(.5,.75);
\draw[thick] (.5,-.5)--(.5,-.75);
\node at (0,-1.2) {$\;$};
\end{tikzpicture}
\qquad 
\begin{tikzpicture}[x=.35cm, y=.35cm,
    every edge/.style={
        draw,
      postaction={decorate,
                    decoration={markings}
                   }
        }
]

\node at (-1.25,-3.25) {\;};

\draw[thick] (2,2)--(1,3)--(-1,1);
\draw[thick] (-1,-1)--(1,-3)--(2,-2);

\draw[thick] (2,2)--(4,0)--(2,-2);
 \draw[thick] (0,2) -- (2,0)--(3,-1);
 \draw[thick,red] (0,0) -- (-1,1)--(-2,-0)--(-1,-1)--(0,0);

 \draw[thick] (1,3)--(1,3.5);
 \draw[thick] (1,-3)--(1,-3.5);

\end{tikzpicture}
\]
\caption{A pair of opposing carets and a pair of trees equivalent to $x_0$.} \label{caret}
\end{figure}
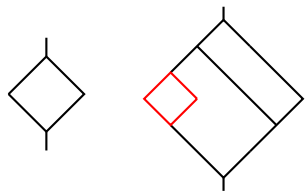

The shift homomorphism  $\varphi : F\to  F$ is an injective group homomorphism defined graphically  as
\[
\begin{tikzpicture}[x=1cm, y=1cm,
    every edge/.style={
        draw,
      postaction={decorate,
                    decoration={markings}
                   }
        }
]

\node at (-.45,0) {$\scalebox{1}{$\varphi$:}$};

\draw[thick] (0,0)--(.5,.5)--(1,0)--(.5,-.5)--(0,0);
\node at (1.5,0) {$\scalebox{1}{$\mapsto$}$};

\node at (.5,0) {$\scalebox{1}{$g$}$};
\node at (.5,.-.75) {$\scalebox{1}{}$};

 \draw[thick] (.5,.65)--(.5,.5);
 \draw[thick] (.5,-.65)--(.5,-.5);
 
\end{tikzpicture}
\begin{tikzpicture}[x=1cm, y=1cm,
    every edge/.style={
        draw,
      postaction={decorate,
                    decoration={markings}
                   }
        }
]
 
\node at (2.5,0) {$\scalebox{1}{$g$}$};

\draw[thick] (1.5,0)--(2.25,.75)--(3,0)--(2.25,-.75)--(1.5,0);
\draw[thick] (2,0)--(2.5,.5)--(3,0)--(2.5,-.5)--(2,0);
 
 \node at (1.5,.-.75) {$\scalebox{1}{}$};


 \draw[thick] (2.25,.75)--(2.25,.9);
 \draw[thick] (2.25,-.75)--(2.25,-.9);

\end{tikzpicture}
\]
This homomorphism maps $x_i$ to $x_{i+1}$ for all $i\geq 0$. The range of the   shift homomorphism is the subgroup of $F$
 consisting of the elements that act trivially on $[0,1/2]$. 

We now recall two different and equivalent descriptions of Jones's construction 
  of knots and links from elements of $F$.
  
  \bigskip
  \noindent
\textbf{First method}.  
We  illustrate this method with $x_0x_1$ as an example.
The idea is to construct a Tait diagram $\Gamma(T_+,T_-)$ from a binary tree diagram $(T_+,T_-)$ in $F$.
We put the vertices of $\Gamma(T_+,T_-)$ on the half integers. 
For $x_0x_1$ these points are $(1/2,0)$, $(3/2,0)$, $(5/2,0)$, $(7/2,0)$.
The edges of $\Gamma(T_+,T_-)$ pass transversally through the edges of the top tree sloping up from left to right (we call them West-North edges, or simply WN$=$\rotatebox[origin=tr]{-45}{|})
and the edges of the bottom tree sloping down from left to right (we refer to them by West-South edges, or just  WS$=$\rotatebox[origin=tr]{45}{|}).
\[
\begin{tikzpicture}[x=.35cm, y=.35cm,
    every edge/.style={
        draw,
      postaction={decorate,
                    decoration={markings}
                   }
        }
]

\node at (-4,0) {$\scalebox{1}{$x_0x_1=$}$};
\node at (-1.25,-3) {\;};

\draw[thick] (0,0) -- (3,3)--(6,0)--(3,-3)--(0,0);
 \draw[thick] (3,1) -- (2,0)--(4,-2);

 \draw[thick] (2,2)--(5,-1);

 \draw[thick] (3,-3)--(3,-3.5);
 \draw[thick] (3,3)--(3,3.5);


\draw[thick, red] (-1,0) to[out=90,in=90] (1,0);
\draw[thick, red] (-1,0) to[out=90,in=180] (2,2.25);
\draw[thick, red] (2,2.25) to[out=0,in=90] (5,0);
\draw[thick, red] (1,0) to[out=90,in=90] (3,0);
\draw[thick, red] (-1,0) to[out=-90,in=-90] (1,0);
\draw[thick, red] (1,0) to[out=-90,in=-90] (3,0);
\draw[thick, red] (3,0) to[out=-90,in=-90] (5,0);

\fill (-1,0)  circle[radius=1.5pt];
\fill (1,0)  circle[radius=1.5pt];
\fill (3,0)  circle[radius=1.5pt];
\fill (5,0)  circle[radius=1.5pt];
\end{tikzpicture}
\qquad
\begin{tikzpicture}[x=.35cm, y=.35cm,
    every edge/.style={
        draw,
      postaction={decorate,
                    decoration={markings}
                   }
        }
] 

 \node at (-3.75,0) {$\scalebox{1}{$\Gamma(x_0x_1)=$}$};
 \node at (-2.5,-3) {$\scalebox{1}{$\,$}$};

\draw[thick, red] (-1,0) to[out=90,in=90] (1,0);
\draw[thick, red] (-1,0) to[out=90,in=90] (5,0);
\draw[thick, red] (1,0) to[out=90,in=90] (3,0);
\draw[thick, red] (-1,0) to[out=-90,in=-90] (1,0);
\draw[thick, red] (1,0) to[out=-90,in=-90] (3,0);
\draw[thick, red] (3,0) to[out=-90,in=-90] (5,0);

\fill (-1,0)  circle[radius=1.5pt];
\fill (1,0)  circle[radius=1.5pt];
\fill (3,0)  circle[radius=1.5pt];
\fill (5,0)  circle[radius=1.5pt];

\end{tikzpicture}
\]
As shown in \cite[Lemma 4.1.4]{Jo14} there is a bijection between the graphs of  the form $\Gamma(T_+,T_-)$ and the pairs of trees $(T_+,T_-)$.
We denote by $\Gamma_+(T_+)$ and $\Gamma_-(T_-)$ the subgraphs of $\Gamma(T_+,T_-)$ contained in the upper and lower-half plane, respectively. Sometimes, to ease the notation we will use the symbols $\Gamma_+$ and $\Gamma_-$.  Since a Tait diagram is a signed graph, we say that the edges of $\Gamma_+$ (resp. $\Gamma_-$) are positive (resp. negative). 

\begin{remark} \label{remarkGamma}
The graphs of the type $\Gamma(T_+,T_-)$ may always be assumed to satisfy the following properties
\begin{enumerate}
\item the vertices are $(0,0)$, \ldots , $(N,0)$; 
\item each vertex other than $(0,0)$  is connected to exactly one vertex to its left;
\item each edge $e$ can be parametrized by a function $(x_e(t),y_e(t))$ with $x'_e(t)>0$, for all $t\in [0,1]$, and either $y_e(t)>0$, for all $t\in ]0,1[$ or $y_e(t)<0$, for all $t\in ]0,1[$;
\end{enumerate}
see \cite[Proposition 4.1.3.]{Jo14} and \cite[Section 3]{A}. In particular, every vertex (except the leftmost) is the target of exactly two edges, one in the lower half-plane and one in the upper-half plane.
\end{remark}

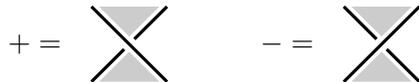
\begin{figure}
 \[\begin{tikzpicture}[every path/.style={very thick}, every node/.style={transform shape, knot crossing, inner sep=1.5pt}]

\node (aaaa) at (-.75,0.5) {$\scalebox{1}{$+=$}$}; 

\node (a1) at (0,0) {}; 
\node (a2) at (0,1) {};
\node (a3) at (1,1) {};
\node (a4) at (1,0) {};
\node (a5) at (0.5,0.5) {};
%
%
\draw (a1.center) .. controls (a1.4 north east) .. (a5) .. controls (a5.4 north east) .. (a3.center);
\draw (a2.center) .. controls (a2.4 south east) .. (a5.center) .. controls (a5.4 south east) .. (a4.center);

  \fill [color=black,opacity=0.2]
               (0.1,0) -- (.5,.4) -- (.9,0) --(.1, 0);

  \fill [color=black,opacity=0.2]
               (0.1,1) -- (.5,.6) -- (.9,1) --(.1, 1); 
\end{tikzpicture}
\qquad 
\qquad\qquad
\begin{tikzpicture}[every path/.style={very thick}, every node/.style={transform shape, knot crossing, inner sep=1.5pt}]

\node (aaaa) at (1.5,0.5) {$\scalebox{1}{$-=$}$}; 

\node (b1) at (2.25,0) {}; 
\node (b2) at (2.25,1) {};
\node (b3) at (3.25,1) {};
\node (b4) at (3.25,0) {};
\node (b5) at (2.75,0.5) {};
 
\draw (b1.center) .. controls (b1.4 north east) .. (b5.center) .. controls (b5.4 north east) .. (b3.center);
\draw (b2.center) .. controls (b2.4 south east) .. (b5) .. controls (b5.4 south east) .. (b4.center);

  \fill [color=black,opacity=0.2]
               (2.35,0) -- (2.75,.4) -- (3.15,0) --(2.35, 0); 

  \fill [color=black,opacity=0.2]
               (2.35,1) -- (2.75,.6) -- (3.15,1) --(2.35, 1);

\end{tikzpicture}\] 
\caption{A positive and a negative crossing. }\label{signscross}
\end{figure}

In order to obtain a knot diagram we need two further steps. First we draw the medial graph $M(\Gamma(T_+,T_-))$ of $\Gamma(T_+,T_-)$. 
In general, given a connected plane graph $G$,
the vertices of its medial graph $M(G)$ sit on every edge of $G$ and an edge of $M(G)$ connect two vertices if they are on adjacent edges of the same face.
Below we will provide an example in our context.
Now all the vertices of $M(\Gamma(T_+,T_-))$ have degree $4$ and to obtain a knot/link diagram we need to turn the vertices into crossings. 
For the vertices in the upper-half plane we use the crossing $\slashoverback$,
while for those in the lower-half plane we use $\backoverslash$. 

Recall that given a link diagram $L$ we may shade it in black and white (we adopt the convention that the colour of the unbounded region is white). This yields a surface in $\mathbb{R}^3$ whose boundary is the link $L$. This procedure is called checkerboard shading of the link diagram.

We point out that in the 
checkerboard shading of the link diagram 
$L(T_+,T_-)$
obtained above, 
the crossings corresponding to vertices on edges of $\Gamma_+$ are positive and the crossings corresponding to vertices on edges of $\Gamma_-$
 are negative (in the sense of Figure \ref{signscross}), see \cite[Section 5.3.2]{Jo14}.
Here are $M(\Gamma(x_0x_1))$ and $\CL(x_0x_1)$.
\[
\phantom{This text} 
\begin{tikzpicture}[x=.6cm, y=.6cm, every path/.style={
 thick}, 
every node/.style={transform shape, knot crossing, inner sep=1.5pt}]
\node (aaaa) at (-3,1) {$\scalebox{1}{$M(\Gamma(T_+,T_-))=$}$}; 

\node (a2) at (1,0) {};
\node (a3) at (2,0) {};
\node (a4) at (3,0) {};
\node (b2) at (1,1) {};
\node (b3) at (2,1) {};
\node (b4) at (3,2) {};

\node (appo) at (1,2) {};

\draw (a2.center) .. controls (a2.4 north west) and (b2.4 south west) .. (b2.center);
\draw (a2.center) .. controls (a2.4 north east) and (b2.4 south east) .. (b2.center);

\draw (b2.center) .. controls (b2.8 north west) and (b4.4 south west) .. (b4.center);

\draw (b2.center) .. controls (b2.4 north east) and (b3.4 north west) .. (b3.center);
\draw (a2.center) .. controls (a2.4 south east) and (a3.4 south west) .. (a3.center);
\draw (a3.center) .. controls (a3.4 north east) and (b3.4 south east) .. (b3.center);
\draw (a3.center) .. controls (a3.4 north west) and (b3.4 south west) .. (b3.center);
\draw (a2.center) .. controls (a2.16 south west) and (appo.8 west) .. (appo.center);
\draw (appo.center) .. controls (appo.4 east) and (b4.4 north west) .. (b4.center);

\draw (a3.center) .. controls (a3.4 south east) and (a4.4 south west) .. (a4.center);
\draw (a4.center) .. controls (a4.16 south east) and (b4.16 north east) .. (b4.center);
\draw (a4.center) .. controls (a4.4 north east) and (b4.4 south east) .. (b4.center);

\draw (b3.center) .. controls (b3.4 north east) and (a4.4 north west) .. (a4.center);

\end{tikzpicture}
\qquad\qquad\qquad\qquad 
\begin{tikzpicture}[x=.6cm, y=.6cm, every path/.style={
 thick}, 
every node/.style={transform shape, knot crossing, inner sep=1.5pt}]

\node (aaaa) at (-3,1) {$\scalebox{1}{$\CL(T_+,T_-)=$}$}; 

\node (a2) at (1,0) {};
\node (a3) at (2,0) {};
\node (a4) at (3,0) {};
\node (b2) at (1,1) {};
\node (b3) at (2,1) {};
\node (b4) at (3,2) {};

\node (appo) at (1,2) {};

\draw (a2.center) .. controls (a2.4 north west) and (b2.4 south west) .. (b2.center);
\draw (a2) .. controls (a2.4 north east) and (b2.4 south east) .. (b2);

\draw (b2) .. controls (b2.8 north west) and (b4.4 south west) .. (b4.center);

\draw (b2.center) .. controls (b2.4 north east) and (b3.4 north west) .. (b3);
\draw (a2.center) .. controls (a2.4 south east) and (a3.4 south west) .. (a3);
\draw (a3) .. controls (a3.4 north east) and (b3.4 south east) .. (b3);
\draw (a3.center) .. controls (a3.4 north west) and (b3.4 south west) .. (b3.center);
\draw (a2) .. controls (a2.16 south west) and (appo.8 west) .. (appo.center);
\draw (appo.center) .. controls (appo.4 east) and (b4.4 north west) .. (b4);

\draw (a3.center) .. controls (a3.4 south east) and (a4.4 south west) .. (a4);
\draw (a4.center) .. controls (a4.16 south east) and (b4.16 north east) .. (b4.center);
\draw (a4) .. controls (a4.4 north east) and (b4.4 south east) .. (b4);

\draw (b3.center) .. controls (b3.4 north east) and (a4.4 north west) .. (a4.center);

\end{tikzpicture}
\]
\begin{remark}
We point out that there are slightly different notations in the literature.
For instance in \cite[Definition 3.4]{Jo18} and \cite{ACJ} the link diagram $L(g)$ denotes   
 the unique representation of $g$ as a pair of binary trees with minimal number of vertices
(similarly in \cite{ACJ},   $\vec{\CL}(g)$ denotes the oriented link diagram produced by the minimal representative of $g\in\vec F$).
In the present article we prefer to denote $g$ explicitly by a pair of binary trees $(T_+,T_-)$ in order to emphasise the role of the trees (this is in the same spirit as in \cite[Definition 4.1.2]{Jo14}, where the Tait diagrams
of the link constructed from the minimal representatives of $g$ can either be denoted by $\Gamma(T_+,T_-)$ or $\Gamma(g)$). 
A slightly different notation is employed in \cite{GS} because a different description of $F$ as a diagram group is used.
\end{remark}
So far we have obtained an unoriented knot/link.
In general the link diagrams obtained from elements of $F$
do not admit a natural orientation.
However, there is a natural orientation when the group element is in the oriented Thompson group $\vec{F}$, whose definition we now recall.
Shade the   link diagram $\CL(T_+,T_-)$ in black and white 
 (the unbounded region is white). This yields a surface in $\mathbb{R}^3$ whose boundary is the link $\mathcal{L}(T_+,T_-)$ (see \cite[Section 5.3.2]{Jo14}). 
The oriented Thompson group $\vec{F}$ can be defined as 
\begin{align*}
\vec F&:=\{(T_+,T_-)\in F\; | \; \Gamma(T_+,T_-) \text{ is bipartite} \}\; .
\end{align*}
Equivalently, the elements of $\vec{F}$ have Tait diagram  $2$-colorable.
We denote the colours by $\{+, -\}$.
We recall that if the Tait graph is $2$-colorable, then there are exactly two colorings. By convention we choose the one
in which the leftmost vertex is assigned the colour $+$. We denote by $\vec{F}_+$ the monoid $\vec{F}\cap F_+$.

By construction the vertices of the graph $\Gamma(T_+, T_-)$ sit in the black regions and each one has been assigned with a colour $+$ or $-$. 
These colours determine an orientation of the surface and of the boundary ($+$ means that the region is positively oriented). 
It can be easily seen that the graph $\Gamma(x_0x_1)$ is bipartite and thus $x_0x_1$ is in $\vec{F}$ (this element is actually one of the three natural generators of $\vec{F}$). 
Here is the oriented link associated with $x_0x_1$.
\[
\begin{tikzpicture}[x=.6cm, y=.6cm, every path/.style={
 thick}, 
every node/.style={transform shape, knot crossing, inner sep=1.5pt}]

\node (aaaa) at (-3,1) {$\scalebox{1}{$\vec\CL(T_+,T_-)=$}$}; 

\node (a2) at (1,0) {};
\node (a3) at (2,0) {};
\node (a4) at (3,0) {};
\node (b2) at (1,1) {};
\node (b3) at (2,1) {};
\node (b4) at (3,2) {};

\node (appo) at (1,2) {};

\draw[->] (a2.center) .. controls (a2.4 north west) and (b2.4 south west) .. (b2.center);
\draw[->] (a2) .. controls (a2.4 north east) and (b2.4 south east) .. (b2);

\draw[->] (b2) .. controls (b2.8 north west) and (b4.4 south west) .. (b4.center);

\draw[->] (b2.center) .. controls (b2.4 north east) and (b3.4 north west) .. (b3);
\draw (a2.center) .. controls (a2.4 south east) and (a3.4 south west) .. (a3);
\draw (a3) .. controls (a3.4 north east) and (b3.4 south east) .. (b3);
\draw[<-] (a3.center) .. controls (a3.4 north west) and (b3.4 south west) .. (b3.center);
\draw[<-] (a2) .. controls (a2.16 south west) and (appo.8 west) .. (appo.center);
\draw (appo.center) .. controls (appo.4 east) and (b4.4 north west) .. (b4);

\draw (a3.center) .. controls (a3.4 south east) and (a4.4 south west) .. (a4);
\draw (a4.center) .. controls (a4.16 south east) and (b4.16 north east) .. (b4.center);
\draw[->] (a4) .. controls (a4.4 north east) and (b4.4 south east) .. (b4);

\draw (b3.center) .. controls (b3.4 north east) and (a4.4 north west) .. (a4.center);

 \node (x2) at (.5,.5) {$\scalebox{0.75}{$+$}$};
\node (x3) at (1.45,.5) {$\scalebox{0.75}{$-$}$};
\node (x4) at (2.5,.5) {$\scalebox{0.75}{$+$}$};
\node (x5) at (3.4,.5) {$\scalebox{0.75}{$-$}$};
 
\end{tikzpicture} 
\]

The following result will come in handy in the next section.
\begin{lemma}\label{lemmasignoriented}
For any $(T_+,T_-)\in \vec{F}$, the crossings of the oriented link $\vec{\CL}(T_+,T_-)$ corresponding to the edges of $\Gamma_+$ ($\, \Gamma_-$) are positive (negative\footnote{Here positive and negative are understood in the sense of oriented link diagrams, see Figure \ref{figpos}.}, respectively).
\end{lemma}
\begin{proof}
As explained in Remark \ref{remarkGamma}, each vertex of $\Gamma(T_+,T_-)$ (except the first one) is the the target of two vertices: one in the upper-half plane, one in the lower-half plane.
We now look at   each target vertex and the two edges.
Below on the left we show the case where the vertex has been assigned color $+$, while on the right is the case where the color is $-$.
\[
\begin{tikzpicture}[x=.6cm, y=.6cm, every path/.style={
 thick}, 
every node/.style={transform shape, knot crossing, inner sep=1.5pt}]


\node (a1) at (0,0) {};
\node (b1) at (0,2) {}; 

\node (appo) at (1,2) {};

\draw[<-] (a1) .. controls (a1.4 north east) and (b1.4 south east) .. (b1);
\draw[->] (-.25,.25)--(.25,-.25);
\draw (-.25,-.25)--(-.05,-.05);

\draw[<-] (-.25,1.75)--(.25,2.25);
 \draw[->] (-.25,2.25)--(-.05,2.05);

 \node (x2) at (1,.95) {$\scalebox{0.75}{$+$}$}; 
 
\end{tikzpicture} 
\qquad\qquad\qquad\qquad
\begin{tikzpicture}[x=.6cm, y=.6cm, every path/.style={
 thick}, 
every node/.style={transform shape, knot crossing, inner sep=1.5pt}]


\node (a1) at (0,0) {};
\node (b1) at (0,2) {}; 

\node (appo) at (1,2) {};

\draw[->] (a1) .. controls (a1.4 north east) and (b1.4 south east) .. (b1);
\draw[<-] (-.25,.25)--(.25,-.25);
\draw (-.25,-.25)--(-.05,-.05);

\draw[->] (-.25,1.75)--(.25,2.25);
 \draw[<-] (-.25,2.25)--(-.05,2.05);

 \node (x2) at (1,.95) {$\scalebox{0.75}{$-$}$}; 
 
\end{tikzpicture} 
\]
\end{proof}

\noindent
\textbf{Second method}.
The construction of the underlying unoriented links from elements of the Thompson group can also be obtained in the following equivalent way, \cite{Jo14, Jo18}.
Starting from a tree diagram in $F$, first we turn all the $3$-valent vertices into $4$-valent and join the two roots of the trees, then we turn all the $4$-valent vertices into crossings 
as shown in Figure \ref{rulesknot}, where these vertices and the four incident edges are replaced by "forks".
We exemplify this procedure with $x_0x_1$. 
\begin{figure}
\[
\begin{tikzpicture}[x=.3cm, y=.3cm,
    every edge/.style={
        draw,
      postaction={decorate,
                    decoration={markings}
                   }
        }
]

\draw[thick] (1,1)--(1,2);
\draw[thick] (0,0) --(1,1)--(2,0);

\node at (0,-1.2) {$\;$};

\end{tikzpicture}\quad
\begin{tikzpicture}[x=.3cm, y=.3cm,
    every edge/.style={
        draw,
      postaction={decorate,
                    decoration={markings}
                   }
        }
]

\draw[thick] (1,0)--(1,2);
\draw[thick] (0,0) --(1,1)--(2,0);

\node at (0,-1.2) {$\;$};
\node at (-2,1) {$\scalebox{1}{$\mapsto\; $}$};

\end{tikzpicture}\qquad
\begin{tikzpicture}[x=.3cm, y=.3cm,
    every edge/.style={
        draw,
      postaction={decorate,
                    decoration={markings}
                   }
        }
]

\draw[thick] (1,0)--(1,2);
\draw[thick] (0,0) --(1,1)--(2,0);

\node at (0,-1.2) {$\;$};

\end{tikzpicture}\quad
\begin{tikzpicture}[x=.3cm, y=.3cm,
    every edge/.style={
        draw,
      postaction={decorate,
                    decoration={markings}
                   }
        }
]

\draw[thick] (1,0)--(1,.35);
\draw[thick] (1,.75)--(1,2);
\draw[thick] (0,0) to[out=90,in=90] (2,0);

\node at (-2,1) {$\scalebox{1}{$\mapsto\; $}$};

\node at (0,-1.2) {$\;$};

\end{tikzpicture}
\qquad \begin{tikzpicture}[x=.3cm, y=.3cm,
    every edge/.style={
        draw,
      postaction={decorate,
                    decoration={markings}
                   }
        }
]

\draw[thick] (1,0)--(1,2);
\draw[thick] (0,2) --(1,1)--(2,2);

\node at (0,-1.2) {$\;$};

\end{tikzpicture}\quad
\begin{tikzpicture}[x=.3cm, y=.3cm,
    every edge/.style={
        draw,
      postaction={decorate,
                    decoration={markings}
                   }
        }
]

\draw[thick] (1,2)--(1,1.65);
\draw[thick] (1,1.25)--(1,0);
\draw[thick] (0,2) to[out=-90,in=-90] (2,2);

\node at (-2,1) {$\scalebox{1}{$\mapsto\; $}$};

\node at (0,-1.2) {$\;$};

\end{tikzpicture}
\]
\caption{The rules needed for obtaining $\CL(g)$. }\label{rulesknot}
\end{figure}
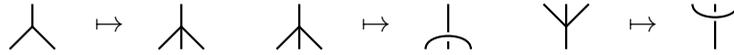
\[
\begin{tikzpicture}[x=.35cm, y=.35cm,
    every edge/.style={
        draw,
      postaction={decorate,
                    decoration={markings}
                   }
        }
]
 
\draw[thick] (0,0) -- (3,3)--(6,0)--(3,-3)--(0,0);
 \draw[thick] (3,1) -- (2,0)--(4,-2);

 \draw[thick] (2,2)--(5,-1);

 \draw[thick] (3,-3)--(3,-3.5);
 \draw[thick] (3,3)--(3,3.5);

\node at (7,0) {$\scalebox{1}{$\mapsto$}$};

\node at (0,-1.2) {\phantom{$\frac{T_+}{T_-}=$}};
\node at (0,-3) {\phantom{$\frac{T_+}{T_-}=$}};

\end{tikzpicture}
%
%
%
%
%
%
%
%
\begin{tikzpicture}[x=.35cm, y=.35cm,
    every edge/.style={
        draw,
      postaction={decorate,
                    decoration={markings}
                   }
        }
]

 \draw[thick] (3,3.5) -- (3,3);  
 \draw[thick] (3,-3.5) -- (3,-3);  

\draw[thick] (0,0) -- (3,3)--(6,0)--(3,-3)--(0,0);
 \draw[thick] (3,1) -- (2,0)--(4,-2);

 \draw[thick] (2,2)--(5,-1);

\node at (7,0) {$\scalebox{1}{$\mapsto$}$};

 \draw[thick] (2,2)--(1,0)--(3,-3);
 \draw[thick] (3,1)--(3,0)--(4,-2);
 \draw[thick] (5,-1)--(5,0)--(3,3);

\node at (0,-1.2) {\phantom{$\frac{T_+}{T_-}=$}};
\node at (0,-3) {\phantom{$\frac{T_+}{T_-}=$}};

\end{tikzpicture}
\begin{tikzpicture}[x=.35cm, y=.35cm,
    every edge/.style={
        draw,
      postaction={decorate,
                    decoration={markings}
                   }
        }
]

 \draw[thick] (-1,3) to[out=90,in=90] (3,3);  
 \draw[thick] (-1,-3) to[out=-90,in=-90] (3,-3);  
\draw[thick] (-1,-3) --(-1,3);

\draw[thick] (0,0) -- (3,3)--(6,0)--(3,-3)--(0,0);
 \draw[thick] (3,1) -- (2,0)--(4,-2);

 \draw[thick] (2,2)--(5,-1);

\node at (7,0) {$\scalebox{1}{$\mapsto$}$};

 \draw[thick] (2,2)--(1,0)--(3,-3);
 \draw[thick] (3,1)--(3,0)--(4,-2);
 \draw[thick] (5,-1)--(5,0)--(3,3);

\node at (0,-1.2) {\phantom{$\frac{T_+}{T_-}=$}};
\node at (0,-3) {\phantom{$\frac{T_+}{T_-}=$}};

\end{tikzpicture}
\begin{tikzpicture}[x=.35cm, y=.35cm,
    every edge/.style={
        draw,
      postaction={decorate,
                    decoration={markings}
                   }
        }
]

\draw[thick] (2,0) to[out=90,in=90] (4,0);   
\draw[thick] (0,0) to[out=90,in=90] (3,.75);   
\draw[thick] (1,1.25) to[out=90,in=90] (6,0);   
\draw[thick] (3,2.4) to[out=90,in=90] (-1,0);   
\draw[thick] (1,.9)--(1,0);
\draw[thick] (3,.4)--(3,0);
\draw[thick] (3,2.1) to[out=-90,in=90] (5,0);

\draw[thick] (4,0) to[out=-90,in=-90] (6,0);   
\draw[thick] (2,0) to[out=-90,in=-90] (5,-.8);   
\draw[thick] (0,0) to[out=-90,in=-90] (3,-1.3);   
\draw[thick] (-1,0) to[out=-90,in=-90] (2,-2);   
\draw[thick] (5,0)--(5,-.4);
\draw[thick] (3,0)--(3,-.9);
\draw[thick] (1,0) to[out=-90,in=90] (2,-1.6);
 

\node at (0,-1.2) {$\;$};

\end{tikzpicture}
\qquad
\]
In the first step (when we turn binary tree diagrams into ternary tree diagrams)
we are actually using an injective group homomorphism $\phi: F\to F_3$ defined by Jones in \cite[Section 4]{Jo18}.

The Tait diagram of the link diagram (see \cite[Chapter X]{Reide1}) obtained in this way is exactly the one described with the previous procedure.
When we consider elements of $\vec{F}$,
the coloring of this graph determines the orientation of the link.
In passing, we mention that by means of planar algebras \cite{jo2} and this construction of knots, several unitary representations of both the Thompson group and the oriented Thompson group related to notable knot and graph invariants were defined \cite{Jo14, Jo16, ACJ, ABC, AiCo1, AiCo2} and investigated \cite{AJ, Jo19, TV, TV2}.

We recall that the bottom tree of a positive element in   $\vec{F}_+$ and  the corresponding graph $\Gamma_-$  have the following form 
\begin{eqnarray}\label{Tmeno}
&
\begin{tikzpicture}[x=.6cm, y=.6cm,
    every edge/.style={
        draw,
      postaction={decorate,
                    decoration={markings}
                   }
        }
]
\node (bbb) at (-2,-2) {$\scalebox{1}{$T_-=$}$}; 

\draw[thick] (0,0)--(4,-4)--(8,0);
\draw[thick] (4.5,-3.5)--(1,0);
\draw[thick] (5,-3)--(2,0);
\draw[thick] (7,-1)--(6,0);
\draw[thick] (7.5,-.5)--(7,0);
\draw[thick] (4,-4)--(4,-4.5);

\node (aaaa) at (5,-1) {$\scalebox{1}{$\ldots$}$}; 

\end{tikzpicture}
%
%
%
%
\qquad
\begin{tikzpicture}[x=.85cm, y=.85cm,
    every edge/.style={
        draw,
      postaction={decorate,
                    decoration={markings}
                   }
        }
]
\node (aaaa) at (-1.25,0) {$\scalebox{1}{$\Gamma_-=$}$}; 
\node (aaaa) at (3,0) {$\scalebox{1}{$\ldots$}$}; 
\node (aaaa) at (3,-1.75) {$\scalebox{1}{$\;$}$}; 

\fill (0,0)  circle[radius=1.5pt];
\fill (1,0)  circle[radius=1.5pt];
\fill (2,0)  circle[radius=1.5pt];
\fill (4,0)  circle[radius=1.5pt];
\fill (5,0)  circle[radius=1.5pt];

\draw[thick] (0,0) to[out=-90,in=-90] (1,0);
\draw[thick] (1,0) to[out=-90,in=-90] (2,0);
\draw[thick] (4,0) to[out=-90,in=-90] (5,0);
\end{tikzpicture}
\end{eqnarray}
By convention the coloring of $\Gamma_-$ is $+-+-+-\ldots$. 
Since the bottom tree of a positive element 
depends only on the number of leaves in the upper tree, sometimes we will use the notation $\vec{\CL}(T_+)$, instead of $\vec{\CL}(T_+,T_-)$.

It is well known that every element of the braid group may be expressed as the product of a positive braid and the inverse of a positive braid.
A similar result for the oriented Thompson group was proved by Ren \cite{Ren}:
\begin{proposition}
For any $g\in \vec{F}$, there exists $g_1, g_2\in\vec{F}_{+}$ such that $g=g_1g_2^{-1}$.
\end{proposition}


An oriented knot/link is positive if it admits a knot diagram where all the crossings are positive, see Figure \ref{figpos}.

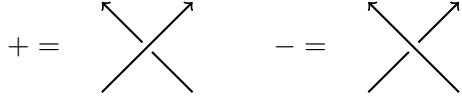
\begin{figure}
\[
\begin{tikzpicture}[x=.6cm, y=.6cm, every path/.style={
 thick}, 
every node/.style={transform shape, knot crossing, inner sep=1.5pt}]

\node (aaaa) at (-1.5,1) {$\scalebox{1}{$+=$}$}; 

\node (a1) at (0,0) {};
\node (a2) at (2,2) {};

\node (b3) at (0,2) {};
\node (b2) at (1,1) {};
\node (b1) at (2,0) {};

\draw[->] (a1.center) .. controls (a1.4 north east) and (a2.4 south west) .. (a2.center);
\draw (b1.center) .. controls (b1.4 north west) and (b2.4 south east) .. (b2);
\draw[->] (b2) .. controls (b2.4 north west) and (b3.4 south east) .. (b3.center);

\end{tikzpicture}
\qquad\qquad\qquad
\begin{tikzpicture}[x=.6cm, y=.6cm, every path/.style={
 thick}, 
every node/.style={transform shape, knot crossing, inner sep=1.5pt}]

\node (aaaa) at (-1.5,1) {$\scalebox{1}{$-=$}$}; 

\node (a1) at (0,0) {};
\node (a2) at (2,2) {};

\node (b3) at (0,2) {};
\node (b2) at (1,1) {};
\node (b1) at (2,0) {};

\draw[->] (b1.center) .. controls (b1.4 north west) and (b3.4 south east) .. (b3.center);
\draw (a1.center) .. controls (a1.4 north east) and (b2.4 south west) .. (b2);
\draw[->] (b2) .. controls (b2.4 north east) and (a2.4 south west) .. (a2.center);

\end{tikzpicture}
\]
\caption{A positive and a negative crossing in an oriented link.} \label{figpos}
\end{figure}

\begin{definition}\label{defF3}
The Brown-Thompson group $F_3$ consists of pairs of rooted planar ternary trees with the same number of leaves, \cite{Brown}.
We call such pairs \emph{ternary tree diagrams}. 
Two  tree diagrams are said to be equivalent if they differ by addition/deletions of pairs of opposing carets, that is 
\[
\begin{tikzpicture}[x=1cm, y=1cm,
    every edge/.style={
        draw,
      postaction={decorate,
                    decoration={markings}
                   }
        }
]

\draw[thick] (0,0)--(.5,.5)--(1,0)--(.5,-.5)--(0,0);
\draw[thick] (.5,.75)--(.5,-.75);
\node at (0,-1.2) {$\;$};
\node at (1.5,0) {$\leftrightarrow$};

\draw[thick] (2.25,.75)--(2.25,-.75);

\end{tikzpicture}
\] 
The Brown-Thompson group $F_3$ may also be defined by the following presentation  
$$
\langle y_0, y_1, \ldots \; | \; y_ny_l=y_ly_{n+2} \quad \forall \; l<n\rangle\, .
$$
The elements $y_0, y_1,   y_{2}$ generate $F_3$.
See Figure \ref{generatorsF3} for a description of these elements in terms of ternary diagram trees. 
The positive elements of $F_3$ are those whose bottom tree can be chosen of  the form 
\[
\begin{tikzpicture}[x=.7cm, y=.7cm,
    every edge/.style={
        draw,
      postaction={decorate,
                    decoration={markings}
                   }
        }
]
\node (bbb) at (-2,-2) {$\scalebox{1}{$T_-'=$}$}; 

\draw[thick] (0,0)--(4,-4)--(8,0);
\draw[thick] (4.5,-3.5)--(1,0);
\draw[thick] (5,-3)--(2,0);
\draw[thick] (7,-1)--(6,0);
\draw[thick] (7.5,-.5)--(7,0);
\draw[thick] (4,-4)--(4,-4.5);

\draw[thick] (0.5,0)--(4,-4);
\draw[thick] (1.5,0)--(4.5,-3.5);
\draw[thick] (2.5,0)--(5,-3);
\draw[thick] (7,-1)--(6.5,0);
\draw[thick] (7.5,-.5)--(7.5,0);

\node (aaaa) at (5,-1) {$\scalebox{1}{$\ldots$}$}; 

\end{tikzpicture}
\]
The monoid consisting of positive elements in $F_3$ is denoted by $F_{3,+}$.
\end{definition}
\begin{figure}
\[
\begin{tikzpicture}[x=.35cm, y=.35cm,
    every edge/.style={
        draw,
      postaction={decorate,
                    decoration={markings}
                   }
        }
]

\node at (-1.5,0) {$\scalebox{1}{$y_0=$}$};
\node at (-1.25,-3) {\;};

\draw[thick] (0,0) -- (2,2)--(4,0)--(2,-2)--(0,0);
\draw[thick] (1,1) -- (1,0)--(2,-2);
\draw[thick] (1,1) -- (2,0)--(3,-1);
\draw[thick] (2,2) -- (3,0)--(3,-1);

 \draw[thick] (2,2)--(2,2.5);
 \draw[thick] (2,-2)--(2,-2.5);

\end{tikzpicture}
\;\;
\begin{tikzpicture}[x=.35cm, y=.35cm,
    every edge/.style={
        draw,
      postaction={decorate,
                    decoration={markings}
                   }
        }
]

\node at (-1.5,0) {$\scalebox{1}{$y_1=$}$};
\node at (-1.25,-3.25) {\;};


\draw[thick] (0,0) -- (2,2)--(4,0)--(2,-2)--(0,0);
 \draw[thick] (1,0)--(2,-2);
\draw[thick] (3,0)--(2,1) -- (1,0); 
\draw[thick] (2,0)--(3,-1);
\draw[thick] (2,2) -- (2,0);
\draw[thick] (3,0)--(3,-1);


 \draw[thick] (2,2)--(2,2.5);
 \draw[thick] (2,-2)--(2,-2.5);
\end{tikzpicture}
\;\;
\begin{tikzpicture}[x=.35cm, y=.35cm,
    every edge/.style={
        draw,
      postaction={decorate,
                    decoration={markings}
                   }
        }
]

\node at (-3.5,0) {$\scalebox{1}{$y_2=$}$};
\node at (-1.25,-3.25) {\;};

\draw[thick] (2,2)--(1,3)--(-2,0)--(1,-3)--(2,-2);

\draw[thick] (0,0) -- (2,2)--(4,0)--(2,-2)--(0,0);
 \draw[thick] (1,1) -- (2,0)--(3,-1);

\draw[thick] (0,0) -- (2,2)--(4,0)--(2,-2)--(0,0);
\draw[thick] (1,1) -- (1,0)--(2,-2);
\draw[thick] (1,1) -- (2,0)--(3,-1);
\draw[thick] (2,2) -- (3,0)--(3,-1);

\draw[thick] (1,3) -- (-1,0)--(1,-3);

\node at (6,0) {$\ldots$};

 \draw[thick] (1,3)--(1,3.5);
 \draw[thick] (1,-3)--(1,-3.5);

\end{tikzpicture}
\]
\caption{The generators of $F_3$.}\label{generatorsF3}
\end{figure}
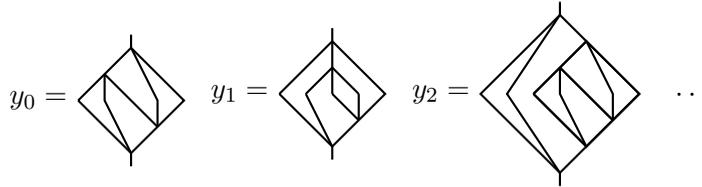

Since the bottom tree of a positive element (either in $F$ and $F_3$) is always of the same form, in the sequel we will only draw the top tree.

In \cite{GS} it was proved that $\vec{F}$ is isomorphic with the Brown-Thompson group $F_3$. 
Later a graphical interpretation of this isomorphism was provided  by Ren in \cite{Ren}: in every ternary tree, the $4$-valent vertices are replaced by a suitable binary tree with $3$ leaves (see Figure \ref{fig-ren-map}). We will use this isomorphism in the next section to study the positive oriented Thompson knots.
Note that the ternary trees of type $T_-'$ (which were introduced in Definition \ref{defF3}) are mapped to the binary trees of type $T_-$ defined in \eqref{Tmeno}. Therefore, by its very definition $\alpha(F_{3,+})$ is contained in $F_+\cap \vec{F}$.

Here follow some examples of positive oriented Thompson knots: (up to disjoint union with unknots) the trefoil, $(3+2n)_2$ twist knot, the $7_4$ knot, the granny knot,  the oriented boundary of an $n$-times twisted annulus and the Pretzel knot $P(3,3,3,3,-2)$.
\begin{figure}
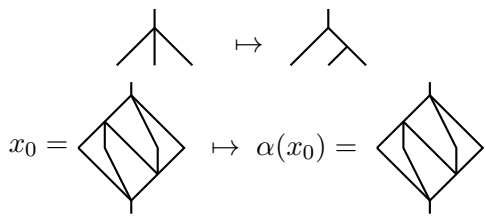

\phantom{This text will be invisible} 
\[

\]
 \caption{The isomorphism $\alpha: F_3\to \vec{F}$ and an example.}  
 \label{fig-ren-map} 
\end{figure}

\begin{example}
The trefoil knot (up to disjoint union with unknots) may be obtained from $g:=x_0^3x_1x_3^2x_4x_{7}^2x_{8}\in\vec{F}_+$.
Here is the top tree of $g$
\[
%
\]
and this is the corresponding Tait diagram
\[
%
\]
\end{example}

\begin{example}
The $(3+2n)_2$ twist knot
 (up to disjoint union with unknots) is obtained with the following element in $\vec F_+$ with top tree $t_n$
\[
%
\]
\end{example}

\begin{example}
The $7_4$ knot (up to disjoint union with unknots) can be obtained from  $g:=x_{0}^3x_{1}x_{3}^2x_{4}x_{6}^2x_{7}x_{9}^2x_{10}x_{13}^2x_{14}x_{16}^2x_{17}x_{19}^2x_{20}\in \vec F_+$ with the following top tree
\[
%
\]
and this is the corresponding Tait diagram
\[
%
\]
\end{example}

\begin{example}
The granny knot (up to disjoint union with unknots) can be obtained from  
$g:=x_0^3x_1x_3^2x_4x_{7}^2x_{8}x_{11}^3x_{12}x_{14}^2x_{15}x_{17}^2x_{19}\in\vec{F}_+$. Below is the top tree of $g$
\[
%
\]
and this is the corresponding Tait diagram
\[
%
\]
\end{example}

\begin{example}
The oriented boundary of an $n$-times twisted annulus  
 (up to disjoint union with unknots) is obtained with the following element in $\vec F_+$ with top tree $t_n$
\[
%
\]
This is the Tait diagram of the link associated with $t_2$ 
\[
%
\]
\end{example}

\begin{example}
The pretzel knot of type $P(3,3,3,3,-2)$ may be obtained (up to disjoint union with unknots) from the element of $\vec{F}_+$ with the following top tree
\[
%
\]
\end{example}
\begin{remark}
The pretzel knot of type $(3,3,3,3,-2)$ is non-alternating, since its minimal diagram is non-alternating.
A general result on the minimality of canonical pretzel link diagrams can be found in \cite{LT}.  
There exist simpler examples of non-alternating positive pretzel knots, for example $P(3,3,-2)$ and $P(5,3,-2)$, also known as torus knots $T(3,4)$ and $T(3,5)$, but we were not able to realise these 
as positive oriented Thompson links.  
The previous example was discovered after  Yuanyuan Bao pointed out to us a mistake in the first version of this paper, where we claimed that  the links corresponding to the elements of $\vec{F}_+$ were alternating.
\end{remark}

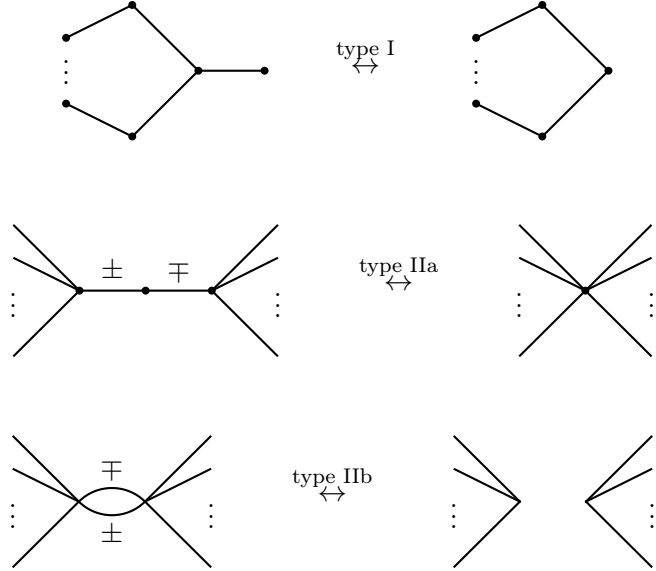
\begin{figure}
\[\begin{tikzpicture}[x=1.75cm, y=1.75cm,
    every edge/.style={
        draw,
      postaction={decorate,
                    decoration={markings}
                   }
        }
]

\node at (0,0.05) {$
\vdots$}; 

\draw[thick, thick] (0,0.25)--(.5,.5)--(1,0)--(.5,-.5)--(0,-0.25);
\draw[thick, thick] (1,0)--(1.5,0);
\fill (0,0.25)  circle[radius=1.5pt];
\fill (0,-0.25)  circle[radius=1.5pt];
\fill (.5,.5)  circle[radius=1.5pt];
\fill (1,0)  circle[radius=1.5pt];
\fill (.5,-.5)  circle[radius=1.5pt];
\fill (1.5,0)  circle[radius=1.5pt];

\end{tikzpicture}
\qquad 
\begin{tikzpicture}[x=2cm, y=2cm,
    every edge/.style={
        draw,
      postaction={decorate,
                    decoration={markings}
                   }
        }
]

\node at (0,.5) {
$\stackrel{\textrm{type I}}{\leftrightarrow}$}; 

\node at (0,0) {$\;$};

\end{tikzpicture}
\qquad
\begin{tikzpicture}[x=1.75cm, y=1.75cm,
    every edge/.style={
        draw,
      postaction={decorate,
                    decoration={markings}
                   }
        }
]

\node at (0,0.05) {
$\vdots$}; 

\draw[thick, thick] (0,0.25)--(.5,.5)--(1,0)--(.5,-.5)--(0,-0.25);
\fill (0,0.25)  circle[radius=1.5pt];
\fill (0,-0.25)  circle[radius=1.5pt];
\fill (.5,.5)  circle[radius=1.5pt];
\fill (1,0)  circle[radius=1.5pt];
\fill (.5,-.5)  circle[radius=1.5pt];

\end{tikzpicture}\]
\newline
\[
\begin{tikzpicture}[x=1.75cm, y=1.75cm,
    every edge/.style={
        draw,
      postaction={decorate,
                    decoration={markings}
                   }
        }
]

\node at (-0.5,-0.05) {
$\vdots$}; 
\node at (1.5,-0.05) {
$\vdots$}; 

\node at (.25,.15) {
$\pm$}; 
\node at (.75,.15) {
$\mp$}; 

\draw[thick] (0,0)--(-.5,.5);
\draw[thick] (0,0)--(-.5,-.5);
\draw[thick] (0,0)--(-.5,0.25);

\draw[thick] (0,0)--(.5,0);
\draw[thick] (.5,0)--(1,0);
\draw[thick] (1.5,0.25)--(1,0);
\draw[thick] (1.5,.5)--(1,0);
\draw[thick] (1.5,-.5)--(1,0);
\fill (0,0)  circle[radius=1.5pt];
\fill (.5,0)  circle[radius=1.5pt];
\fill (1,0)  circle[radius=1.5pt];

\end{tikzpicture}
\qquad 
\begin{tikzpicture}[x=2cm, y=2cm,
    every edge/.style={
        draw,
      postaction={decorate,
                    decoration={markings}
                   }
        }
]

\node at (0,.5) {
$\stackrel{\textrm{type IIa}}{\leftrightarrow}$}; 

\node at (0,0) {$\;$};

\end{tikzpicture}
\qquad
\begin{tikzpicture}[x=1.75cm, y=1.75cm,
    every edge/.style={
        draw,
      postaction={decorate,
                    decoration={markings}
                   }
        }
]

\node at (-0.5,-0.05) {
$\vdots$};
\node at (0.5,-0.05) {
$\vdots$}; 

\draw[thick] (0,0)--(-.5,.5);
\draw[thick] (0,0)--(-.5,-.5);
\draw[thick] (0,0)--(-.5,0.25);

\draw[thick] (.5,0.25)--(0,0);
\draw[thick] (.5,.5)--(0,0);
\draw[thick] (.5,-.5)--(0,0);
\fill (0,0)  circle[radius=1.5pt];

\end{tikzpicture}
\]
\newline
\[
\begin{tikzpicture}[x=1.75cm, y=1.75cm,
    every edge/.style={
        draw,
      postaction={decorate,
                    decoration={markings}
                   }
        }
]

\node at (.25,-.25) {
$\pm$}; 
\node at (.25,.25) {
$\mp$}; 

\draw[thick] (0,0)--(-.5,.5);
\draw[thick] (0,0)--(-.5,-.5);
\draw[thick] (0,0)--(-.5,0.25);

\draw[thick] (.5,0)--(1,0.25);
\draw[thick] (1,.5)--(.5,0);
\draw[thick] (1,-.5)--(.5,0);
\fill (0,0)  circle[radius=.5pt];
\fill (.5,0)  circle[radius=.5pt];

\draw[thick] (0,0) to[out=-45,in=-135] (.5,0);
\draw[thick] (0,0) to[out=45,in=135] (.5,0);

\node at (-0.5,-0.05) {
$\vdots$}; 
\node at (1,-0.05) {
$\vdots$}; 

\end{tikzpicture}
\qquad 
\begin{tikzpicture}[x=2cm, y=2cm,
    every edge/.style={
        draw,
      postaction={decorate,
                    decoration={markings}
                   }
        }
]

\node at (0,.5) {
$\stackrel{\textrm{type IIb}}{\leftrightarrow}$}; 

\node at (0,0) {$\;$};

\end{tikzpicture}
\qquad
\begin{tikzpicture}[x=1.75cm, y=1.75cm,
    every edge/.style={
        draw,
      postaction={decorate,
                    decoration={markings}
                   }
        }
]

\node at (-0.5,-0.05) {
$\vdots$}; 
\node at (1,-0.05) {
$\vdots$}; 

\draw[thick] (0,0)--(-.5,.5);
\draw[thick] (0,0)--(-.5,-.5);
\draw[thick] (0,0)--(-.5,0.25);

\draw[thick] (.5,0)--(1,0.25);
\draw[thick] (1,.5)--(.5,0);
\draw[thick] (1,-.5)--(.5,0);
\fill (0,0)  circle[radius=.5pt];
\fill (.5,0)  circle[radius=.5pt];


\end{tikzpicture}
\]
\caption{
A Reidemeister move of type I allows to add (or remove) a 1-valent vertex and its edge. 
When there is a 2-valent vertex whose edges have opposite signs, they may be contracted as shown in the move of type IIa.
Two parallel edges with  opposite signs may be added (or removed) by means of a move of type IIb. 
}
\label{Reide}
\end{figure}
In Figure \ref{Reide} we display the Reidemeister moves of type I and II in the language of Tait diagrams. 
These moves do not affect the link corresponding to the Tait diagram, see e.g. \cite[Chapter X]{Reide1} and \cite[Figure 3]{Reide2}. 
They will come in handy in the next section.

We conclude the preliminaries with a lemma concerning the structure of the top tree of elements in $\vec{F}_+$. It is a consequence of the proof of \cite[Theorem 5.5]{Ren} and for this reason we only give a sketch of its proof.
\begin{lemma}\label{F3pos}
With the notations of the previous section, it holds that $ \vec{F}_+=\alpha(F_3)\cap \vec{F}_+=\alpha(F_3)\cap F_+=\alpha(F_{3,+})$.
\end{lemma}
\begin{proof}
The inclusion  $\alpha(F_{3,+})\subset \alpha(F_3)\cap \vec{F}_+$ is obvious. 
The converse inclusion can be proved by induction on the number of leaves in the trees and by
showing that  the top tree $T_+$ always contains the following subtree   (the leaves of this subtree are a subset of the leaves of $T_+$)
\[\begin{tikzpicture}[x=.75cm, y=.75cm,
    every edge/.style={
        draw,
      postaction={decorate,
                    decoration={markings}
                   }
        }
]

\draw[thick] (0.5,0.75)--(.5,.5);
\draw[thick] (0.5,0)--(.75,.25);
\draw[thick] (0,0)--(.5,.5)--(1,0);
\end{tikzpicture}
\]
If the leaves of the above tree are the rightmost leaves of $T_+$, then by cancelling two pairs of opposing carets we are done.
Otherwise multiply $g\in \vec{F}_+$ by $(x_ix_{i+1})^{-1}$ (where $i$ is a suitable non-negative integer). 
Note that $x_ix_{i+1}=\alpha(y_i)$ is in $\alpha(F_{3,+})$ and $\vec{F}$, \cite{GS}.
The element $g':=g (x_ix_{i+1})^{-1}$ has less leaves than $g$ and is still in $\alpha(F_3)\cap \vec{F}_+$.
Therefore, by induction it is in $\alpha(F_{3,+})$.
Since $g'=g (x_ix_{i+1})^{-1}\in \alpha(F_{3,+})$ and $x_ix_{i+1}\in \alpha(F_{3,+})$, it follows that $g\in \alpha(F_{3,+})$ as well.
\end{proof}

\section{Positivity of $\vec\CL(\vec{F}_+)$}

In order to prove Theorem~1, we need to transform the oriented link diagrams produced by using the elements of $\vec{F}_+$ into positive ones 
(i.e. we have to remove all the crossings that are negative in the sense of Figure \ref{figpos}). Let $g \in \vec{F}_+$ be represented by a pair of binary trees $(T_+,T_-)$, where $T_-$ is the standard bottom tree defined in \eqref{Tmeno}. By construction, the link diagram $\vec{\CL}(g)$ is a union of two tangles $A$ and $B$, situated above and below the $x$-axis, respectively. Both tangles are alternating\footnote{A tangle is said to be alternating if the crossings alternate from under to over as one moves along any component or arc of the tangle.} and have the same number of crossings $n-1$, where~$n$ is the number of leaves of the two trees $T_+, T_-$. Moreover, by Lemma \ref{lemmasignoriented} all the crossings of~$A$ are positive, while all the crossings of~$B$ are negative\footnote{Here positive and negative should be understood as crossings of an oriented link, as explained in  Figure \ref{figpos}.}, since the Tait graph $\Gamma(T_+,T_-)$ is bipartite. Therefore, in order to obtain a positive diagram, we need to remove all the crossings of the bottom tangle~$B$. 
If we use the second method for constructing links described in Section \ref{secdue}, we see that
there is one negative crossing attached to every string coming out of a leaf of $T_+$, except for the leaf at the very right of the tree, compare Figure \ref{knot5_2}.
We will remove all these negative crossings simultaneously, using two types of local moves. 
By Lemma \ref{F3pos} the binary tree diagram $(T_+,T_-)$ is the image (under the map $\alpha$) of a ternary tree diagram $(T_+',T_-')$, where $T_-'$ is a ternary tree of the form depicted in Definition \ref{defF3}.
The crossings of~$B$ attached to left leaves of the tree $T_+'$ can be dealt with the move described in Figure \ref{leftreduction}, where by means of a Reidemeister move of type II
two crossings (one positive in the upper-half plane
 and one negative in the lower-half plane)
  are removed.
  Similarly, for crossings of ~$B$ attached to middle leaves of the tree $T_+'$, by using a Reidemeister move of type II we remove a positive and a negative crossing (see Figure  \ref{middlereduction}).
In each of these figures 
we depict
a portion of ternary tree diagram, 
a portion of binary tree diagram,  and two tangles.
The small boxes in these figure should be interpreted (respectively) as
portions ternary tree diagram, 
a portion of binary tree diagram,  and two tangles.
The tangles are those produced according to the second procedure described in the previous section.

\begin{figure}[htb] 
\begin{center}
\raisebox{-0mm}{\includegraphics[scale=0.8]{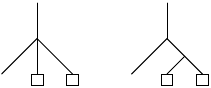}}

\bigskip
\bigskip
\raisebox{-0mm}{\includegraphics[scale=0.8]{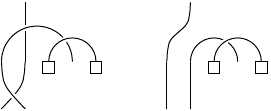}}
\caption{Left leaf.}\label{leftreduction}
\end{center}
\end{figure}

\begin{figure}[htb] 
\begin{center}
\raisebox{-0mm}{\includegraphics[scale=0.8]{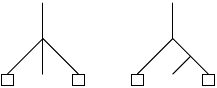}}

\bigskip
\bigskip
\raisebox{-0mm}{\includegraphics[scale=0.8]{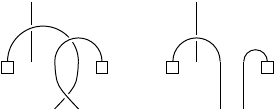}}
\caption{Middle leaf.}\label{middlereduction}
\end{center}
\end{figure}

The crossings of~$B$ attached to right vertices of the tree $T_+'$ can be removed by a detour move, as shown in Figure \ref{rightreduction}. 
This detour move replaces two crossings (one negative and one positive) with a single new positive crossing. 
The precise effect of this move is removing a negative crossing, and replacing one of the crossings of~$A$ by another positive crossing. This can be seen by analysing the auxiliary orientations in the figure (in fact, there are two possibilities for the local orientations; however, the actual choice has no effect on the signs of the crossings). We record two important features about these local moves:
\begin{enumerate}
\item 
they can be performed independently (since they take place in disjoint regions of the diagram),
\item they preserve the positivity of the upper tangle~$A$ and decrease the number of negative crossings in the lower tangle.
\end{enumerate}
In particular, we end up with a positive diagram for the link $\vec{\CL}(g)$. 
This concludes the proof of Theorem 1.

The reader is invited to apply the above procedure to the top diagram of Figure \ref{knot5_2}. The resulting positive diagram is drawn at the bottom of the same figure, with the trivial components removed. 
\begin{remark}
The number of crossings of the final diagram is bounded above by the number of right leaves of the upper tree $T_+'$. 
\end{remark}

\begin{figure}[htb] 
\begin{center}
\raisebox{-0mm}{\includegraphics[scale=0.8]{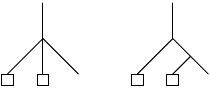}}

\bigskip
\bigskip
\raisebox{-0mm}{\includegraphics[scale=0.8]{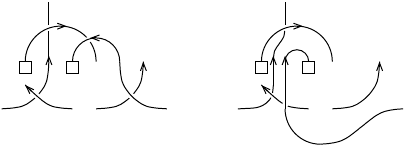}}
\caption{Right leaf.} \label{rightreduction}
\end{center}
\end{figure}


\section{Unknotting positive oriented Thompson links}
The unknotting number of a knot 
is the smallest number of crossing changes
(over all possible diagrams)
  required to transform the knot into 
   the unknot, see e.g. \cite{Licko}.
The next theorem provides an upper bound for the unknotting number of $\vec{\CL}(g)$ when $g$ is in $\vec{F}_+$.
We will use $4$-valent plane trees in order to describe elements of $\vec{F}_+$, as in the proof of Theorem 1.
As we will see, these trees can be reduced to an empty tree by using a set of seven moves as shown in Figure \ref{moves}. 
\begin{theorem}
For any $g\in\vec{F}_+$, the unknotting number is at most equal to the number of applications of  sixth move depicted in Figure \ref{moves}.
\end{theorem}
\begin{proof}
Let $(T_+,T_-)$ be a pair of ternary trees in $F_{3,+}$.
We give a proof by induction on the number $n$ of leaves, 
using
 Lemma \ref{F3pos}. 
When $n=3$, the element $(T_+,T_-)$ is the trivial element of $F_3$ and the corresponding link is trivial.
Now there are seven cases to deal with. Indeed, an easy inductive argument shows that $T_+$ contains one of the subtrees depicted in Figure \ref{moves} (the leaves of this tree are a subset of the leaves of $T_+$).
\begin{figure}
\[
\begin{tikzpicture}[x=.35cm, y=.35cm,
    every edge/.style={
        draw,
      postaction={decorate,
                    decoration={markings}
                   }
        }
]
\node (bbb) at (-2,2) {$\scalebox{1}{$1)$}$}; 

\draw[thick] (1,1)--(4,4)--(8,0);
\draw[thick, red] (0,0)--(1,1);
\draw[thick] (4,4.5)--(4,4);
\draw[thick, red] (2,0)--(1,1);

\draw[thick, red] (1,0)--(1,1);
\draw[thick] (4,0)--(4,4);

\end{tikzpicture}
\qquad
\begin{tikzpicture}[x=.35cm, y=.35cm,
    every edge/.style={
        draw,
      postaction={decorate,
                    decoration={markings}
                   }
        }
]
\node (bbb) at (-2,2) {$\scalebox{1}{$2)$}$}; 

\draw[thick] (0,0)--(4,4)--(8,0);
\draw[thick] (4,4.5)--(4,4);


\draw[thick, red] (4,0)--(4,1);
\draw[thick] (4,1)--(4,4);
\draw[thick, red] (3,0)--(4,1)--(5,0);

\end{tikzpicture}
\qquad
\begin{tikzpicture}[x=.35cm, y=.35cm,
    every edge/.style={
        draw,
      postaction={decorate,
                    decoration={markings}
                   }
        }
]
\node (bbb) at (-2,2) {$\scalebox{1}{$3)$}$}; 

\draw[thick] (0,0)--(4,4)--(7,1);
\draw[thick, red] (7,1)--(8,0);
\draw[thick] (4,4.5)--(4,4);
\draw[thick, red] (6,0)--(7,1);

\draw[thick, red] (7,1)--(7,0);
\draw[thick] (4,0)--(4,4);

\end{tikzpicture}
\]
\[
\begin{tikzpicture}[x=.35cm, y=.35cm,
    every edge/.style={
        draw,
      postaction={decorate,
                    decoration={markings}
                   }
        }
]
\node (bbb) at (-2,2) {$\scalebox{1}{$4)$}$}; 

\draw[thick] (0,0)--(4,4)--(7,1);
\draw[thick, red] (7,1)--(8,0);
\draw[thick] (4,4.5)--(4,4);
\draw[thick, red] (6,0)--(7,1);

\draw[thick, red] (7,1)--(7,0);
\draw[thick, red] (4,0)--(4,1);
\draw[thick] (4,1)--(4,4);
\draw[thick, red] (3,0)--(4,1)--(5,0);

%

\end{tikzpicture}
\qquad
\begin{tikzpicture}[x=.35cm, y=.35cm,
    every edge/.style={
        draw,
      postaction={decorate,
                    decoration={markings}
                   }
        }
]
\node (bbb) at (-2,2) {$\scalebox{1}{$5)$}$}; 

\draw[thick] (1,1)--(4,4)--(7,1);
\draw[thick, red] (0,0)--(1,1);
\draw[thick, red] (7,1)--(8,0);
\draw[thick] (4,4.5)--(4,4);
\draw[thick, red] (2,0)--(1,1);
\draw[thick, red] (6,0)--(7,1);

\draw[thick, red] (1,0)--(1,1);
\draw[thick, red] (7,1)--(7,0);
\draw[thick] (4,0)--(4,4);

%

\end{tikzpicture}
\qquad
\begin{tikzpicture}[x=.35cm, y=.35cm,
    every edge/.style={
        draw,
      postaction={decorate,
                    decoration={markings}
                   }
        }
]
\node (bbb) at (-2,2) {$\scalebox{1}{$6)$}$}; 

\draw[thick, red] (0,0)--(4,4)--(8,0);
\draw[thick] (4,4.5)--(4,4);
\draw[thick, red] (2,0)--(1,1);

\draw[thick, red] (1,0)--(1,1);
\draw[thick, red] (4,0)--(4,4);
\draw[thick, red] (3,0)--(4,1)--(5,0);

%

\end{tikzpicture}
\]
\[
\begin{tikzpicture}[x=.35cm, y=.35cm,
    every edge/.style={
        draw,
      postaction={decorate,
                    decoration={markings}
                   }
        }
]
\node (bbb) at (-2,2) {$\scalebox{1}{$7)$}$}; 

\draw[thick, red] (0,0)--(1,1);
\draw[thick, red] (7,1)--(8,0);
\draw[thick] (1,1)--(4,4)--(7,1);
\draw[thick] (4,4.5)--(4,4);
\draw[thick, red] (2,0)--(1,1);
\draw[thick, red] (6,0)--(7,1);

\draw[thick, red] (1,0)--(1,1);
\draw[thick, red] (7,1)--(7,0);
\draw[thick, red] (4,0)--(4,1);
\draw[thick] (4,1)--(4,4);
\draw[thick, red] (3,0)--(4,1)--(5,0);

%

\end{tikzpicture}
\]
\caption{Seven reduction moves for the elements in $F_{3,+}=\alpha^{-1}(\vec{F}_+)$. 
The red part is the one to be removed.
All these moves except the sixth, do not affect the corresponding knot (only some  unknots are lost in the application of these moves).
In the sixth move a positive crossing is turned into a negative one.}\label{moves}
\end{figure}
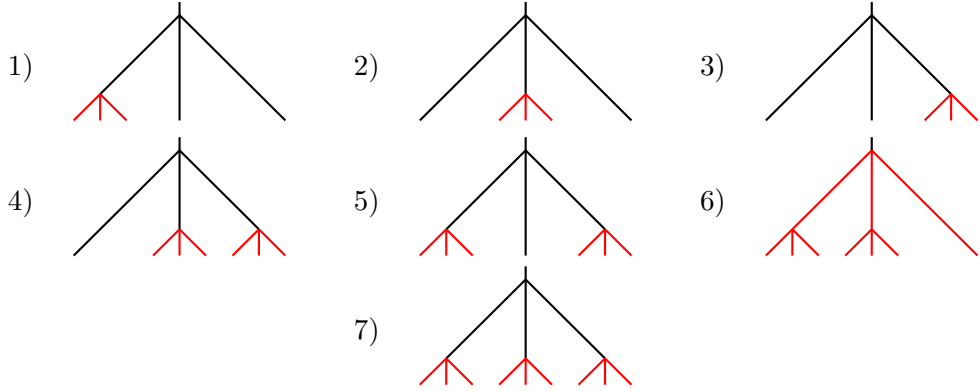
The edges in red in the  figure will be erased.\\
As we shall see, in all the cases, but case 6), we simply apply Reidemeister moves of type I and II, and (possibly) remove unknots. 
In this proof we only make use of the first method for constructing links (see Section \ref{secdue}). In particular, all the Reidemeister moves are described in the framework of
Tait diagrams.
In these cases we often erase pairs of parallel edges. If the sources and targets of these edges both have degree greater than two, 
then we are performing a Reidemeister move of type II. Otherwise, we are removing an unknot. Clearly, removing unknots does not change the unknotting number and, thus, we
are always allowed to erase pairs of parallel edges.

Only in case 6) we need to turn a positive crossing into negative one.

Here follows the subtree of case 1) transformed under the map $\alpha$, the corresponding Tait diagram, an "equivalent" Tait diagram along with the corresponding binary and ternary trees.
 
\[

\]
Here are the analogous graphs for case 2)
\[
%
\] 
In the third case we only need Reidemeister moves of type II as shown below
\[
%
\]
These are the analogous graphs for the fourth case
\[
%
\]
Here is case 5)
\[
%
\]
The sixth case is more complicated. First we simplify the Tait diagram
\[
%
\]
There are two subcases depending on the color of the leftmost vertex. 
Suppose that the color is $+$, then we have the following Tait diagram and the link
\[
%
\]
After turning the leftmost crossing into a negative crossing we get the following link, Tait diagrams, tree
\[
%
\]

Similarly when the color is $-$, then we have the following Tait diagram, the link, Tait diagrams and tree
\[
%
\]
Finally, we  take care of case 7) 
\[
%
\]
\end{proof}
\begin{remark}
The above bound is optimal. Indeed,
let $g$ be $x_0^2x_1x_3^2x_4$, then $\vec \CL(g)$ is the Hopf link (up to disjoint union with unknots). The link $\vec \CL(g\varphi^7(g))$ is a chain link with $3$ connected components.
More generally, $\vec \CL(\prod_{i=0}^{n-1}\varphi^{7i}(g))$ is a chain link with $n$ connected components, see Figure \ref{treechain} (here we are using the notation $\prod_{i=1}^n g_i:=g_1\cdots g_n$).
It is easy to see that the unknotting number and the number of 6-moves needed are both equal to $n$.
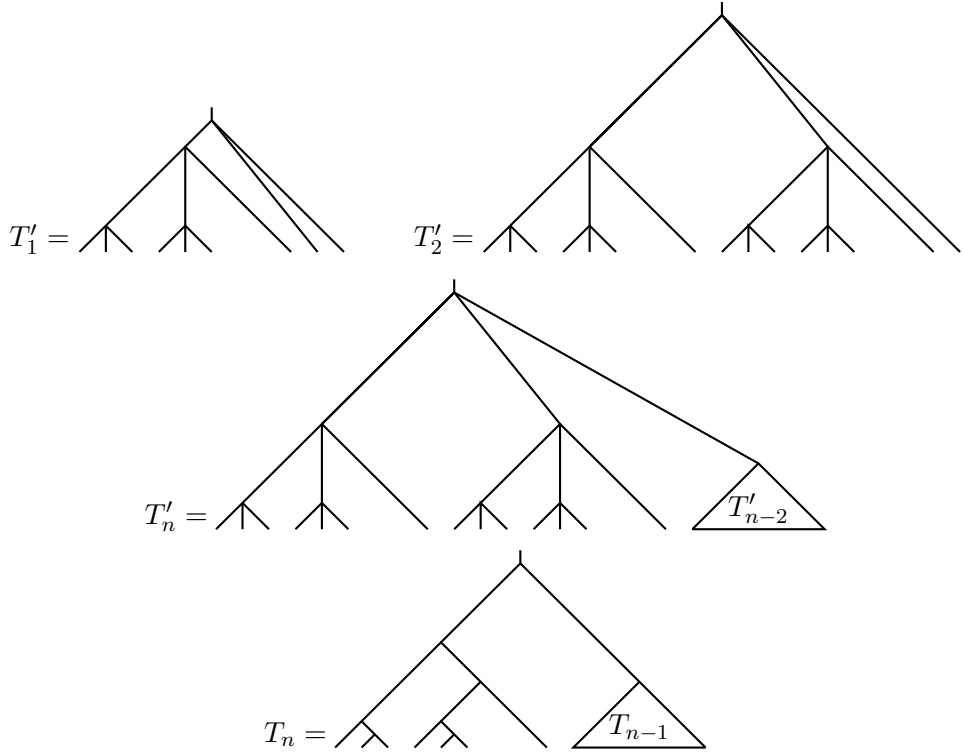
\begin{figure}[h]
\[
\begin{tikzpicture}[x=.35cm, y=.35cm,
    every edge/.style={
        draw,
      postaction={decorate,
                    decoration={markings}
                   }
        }
]

\draw[thick] (0,0)--(4,4)--(8,0);
\draw[thick] (9,0)--(5,5)--(4,4);
\draw[thick] (5,5)--(10,0);
\draw[thick] (5,5)--(5,5.5);
\draw[thick] (2,0)--(1,1);
 
\draw[thick] (1,0)--(1,1);
 \draw[thick] (4,1)--(4,4);
\draw[thick] (4,0)--(4,1);
\draw[thick] (3,0)--(4,1)--(5,0);

\node (a) at (-1.5,.5) {$\scalebox{1}{$T_1'=$}$}; 

\end{tikzpicture}
\qquad
\begin{tikzpicture}[x=.35cm, y=.35cm,
    every edge/.style={
        draw,
      postaction={decorate,
                    decoration={markings}
                   }
        }
]

\draw[thick] (9,0)--(13,4)--(17,0);
 \draw[thick] (12,0)--(13,1)--(14,0);
 \draw[thick] (10,0)--(10,1)--(11,0);

\draw[thick] (0,0)--(4,4)--(8,0);
\draw[thick] (4,4)--(9,9)--(4,4);
\draw[thick] (9,9)--(9,9.5);
\draw[thick] (9,9)--(18,0);
 \draw[thick] (9,9)--(13,4)--(13,0);
\draw[thick] (2,0)--(1,1);
 
\draw[thick] (1,0)--(1,1);
 \draw[thick] (4,1)--(4,4);
\draw[thick] (4,0)--(4,1);
\draw[thick] (3,0)--(4,1)--(5,0);

\node (a) at (-1.5,.5) {$\scalebox{1}{$T_2'=$}$}; 

\end{tikzpicture}
\]
\[
\begin{tikzpicture}[x=.35cm, y=.35cm,
    every edge/.style={
        draw,
      postaction={decorate,
                    decoration={markings}
                   }
        }
]

\draw[thick] (9,0)--(13,4)--(17,0);
 \draw[thick] (12,0)--(13,1)--(14,0);
 \draw[thick] (10,0)--(10,1)--(11,0);

\draw[thick] (9,9)--(20.5,2.5);
\draw[thick] (18,0)--(20.5,2.5)--(23,0)--(18,0);

\draw[thick] (0,0)--(4,4)--(8,0);
\draw[thick] (4,4)--(9,9)--(4,4);
\draw[thick] (9,9)--(9,9.5);
 \draw[thick] (9,9)--(13,4)--(13,0);
\draw[thick] (2,0)--(1,1);
 
\draw[thick] (1,0)--(1,1);
 \draw[thick] (4,1)--(4,4);
\draw[thick] (4,0)--(4,1);
\draw[thick] (3,0)--(4,1)--(5,0);

\node (a) at (-1.5,.5) {$\scalebox{1}{$T_n'=$}$}; 
\node (a) at (20.5,.75) {$\scalebox{1}{$T_{n-2}'$}$}; 

\end{tikzpicture}
\]
\[
%
%
%
%
%
\begin{tikzpicture}[x=.35cm, y=.35cm,
    every edge/.style={
        draw,
      postaction={decorate,
                    decoration={markings}
                   }
        }
]

\draw[thick] (0,0)--(4,4)--(8,0);
\draw[thick] (14,0)--(7,7)--(4,4);
\draw[thick] (14,0)--(9,0)--(11.5,2.5);
\draw[thick] (7,7)--(7,7.5);
\draw[thick] (2,0)--(1,1);
 
\draw[thick] (1,0)--(1.5,.5);
 \draw[thick] (4,1)--(5.5,2.5);
\draw[thick] (4,0)--(4.5,.5);
\draw[thick] (3,0)--(4,1)--(5,0);

\node (a) at (-1.5,.5) {$\scalebox{1}{$T_n=$}$}; 
\node (a) at (11.5,.75) {$\scalebox{1}{$T_{n-1}$}$}; 

\end{tikzpicture}
\]
\caption{
An element of $F_{3,+}$ defined by a recursive relation ($T_0'$ and $T_0$ are the empty trees)
and the corresponding image under the map $\alpha$.
The latter element of $\vec{F}_+$ is  $\prod_{i=0}^{n-1}\varphi^{7i}(g)$ 
and the corresponding link is a chain with $n$ connected components.}\label{treechain}
\end{figure}
%
%
%
%
%
%
\end{remark}   

In the proof of Theorem 2 we saw that elements $g\in \vec{F}_+$
can be simplified by turning positive crossings into negative crossings.
We end this paper with the observation that it is also possible to reduce the "size" of $g$  by smoothing suitable positive crossings.

Each ternary tree diagram in $F_{3,+}$ can also be reduced to an empty tree by using the set of seven moves
1), 2), 3), 4), 5), 6)', 7) shown in 
Figures \ref{moves} and \ref{additional6move}.
\begin{theorem}
For any $g\in\vec{F}_+$, the unknotting number is at most equal to the number of applications of the 
move 
6)' 
 depicted in Figure
  \ref{additional6move}.
\end{theorem}
\begin{proof}
As before we do an inductive argument on the number of  leaves of an element $(T_+,T_-)$ in $F_{3,+}$. 
By the previous discussion, one of the $7$ cases depicted in Figure \ref{moves} occurs.
In all cases but 6)', we may reduce the number of  leaves 
with the only effect of losing some distant unknots.
As in the previous proof, there are two subcases depending on the color of the leftmost vertex. 
Suppose that the color is $+$, then we have the following Tait diagram and the link
\[
\begin{tikzpicture}[x=.35cm, y=.35cm,
    every edge/.style={
        draw,
      postaction={decorate,
                    decoration={markings}
                   }
        }
] 

\node (bbb) at (0,-.5) {$\scalebox{.5}{$+$}$}; 
\node (bbb) at (1,-.5) {$\scalebox{.5}{$-$}$}; 
\node (bbb) at (2,-.5) {$\scalebox{.5}{$+$}$}; 
\node (bbb) at (3,-.5) {$\scalebox{.5}{$-$}$}; 

\fill (0,0)  circle[radius=1.5pt];
\fill (1,0)  circle[radius=1.5pt];
\fill (2,0)  circle[radius=1.5pt];
\fill (3,0)  circle[radius=1.5pt];

\draw[thick] (0,0) to[out=90,in=90] (1,0);
\draw[thick] (1,0) to[out=90,in=90] (2,0);
\draw[thick] (2,0) to[out=-90,in=-90] (3,0);
 
\draw[thick] (0,0) -- (-.5,.5);
\draw[thick] (0,0) -- (-.5,-.5);

\draw[thick] (3,0) -- (2.5,.5);

\draw[thick] (3,0) -- (3.5,.5);
\draw[thick] (3,0) -- (3.5,-.5);

\end{tikzpicture}
\qquad 
\begin{tikzpicture}[x=.3cm, y=.3cm, every path/.style={
 thick}, 
every node/.style={transform shape, knot crossing, inner sep=1.5pt}]


\node (a1) at (0,0) {};
\node (a2) at (2,2) {};

\node (b3) at (0,2) {};
\node (b2) at (1,1) {};
\node (b1) at (2,0) {};

\draw[->] (a1.center) .. controls (a1.4 north east) and (a2.4 south west) .. (a2.center);
\draw (b1.center) .. controls (b1.4 north west) and (b2.4 south east) .. (b2);
\draw[->] (b2) .. controls (b2.4 north west) and (b3.4 south east) .. (b3.center);


\node (a3) at (4,0) {};
\node (a4) at (6,2) {};

\node (b6) at (4,2) {};
\node (b5) at (5,1) {};
\node (b4) at (6,0) {};

\draw[<-] (a3.center) .. controls (a3.4 north east) and (a4.4 south west) .. (a4.center);
\draw (b4.center) .. controls (b4.4 north west) and (b5.4 south east) .. (b5);
\draw (b5) .. controls (b5.4 north west) and (b6.4 south east) .. (b6.center);

\draw (b6.center) .. controls (b6.4 north west) and (a2.4 north east) .. (a2.center);
\draw (a3.center) .. controls (a3.4 south west) and (b1.4 south east) .. (b1.center);


\node (aa1) at (8,-3) {};
\node (aa2) at (10,-1) {};

\node (bb3) at (8,-1) {};
\node (bb2) at (9,-2) {};
\node (bb1) at (10,-3) {};

\draw[->] (bb1.center) .. controls (bb1.4 north west) and (bb3.4 south east) .. (bb3.center);
\draw (aa1.center) .. controls (aa1.4 north east) and (bb2.4 south west) .. (bb2);
\draw[->] (bb2) .. controls (bb2.4 north east) and (aa2.4 south west) .. (aa2.center);

\draw (a4.center) .. controls (a4.4 north east) and (bb3.4 north west) .. (bb3.center);
\draw (aa1.center) .. controls (aa1.4 south west) and (b4.4 south east) .. (b4.center);

\end{tikzpicture}
\]
If we smooth the leftmost crossing,
after applying some Reidemeister moves, we get the following Tait diagram
\[
\begin{tikzpicture}[x=.35cm, y=.35cm,
    every edge/.style={
        draw,
      postaction={decorate,
                    decoration={markings}
                   }
        }
] 

\node (bbb) at (0,-.5) {$\scalebox{.5}{$+$}$}; 
\node (bbb) at (1,-.5) {$\scalebox{.5}{$-$}$}; 
\node (bbb) at (2,-.5) {$\scalebox{.5}{$+$}$}; 
\node (bbb) at (3,-.5) {$\scalebox{.5}{$-$}$}; 

\fill (0,0)  circle[radius=1.5pt];
\fill (1,0)  circle[radius=1.5pt];
\fill (2,0)  circle[radius=1.5pt];
\fill (3,0)  circle[radius=1.5pt];

\draw[thick] (0,0) to[out=90,in=90] (1,0);
\draw[thick] (1,0) to[out=90,in=90] (2,0);
\draw[thick] (0,0) to[out=-90,in=-90] (1,0);
\draw[thick] (1,0) to[out=-90,in=-90] (2,0);
\draw[thick] (2,0) to[out=-90,in=-90] (3,0);
 
\draw[thick] (0,0) -- (-.5,.5);
\draw[thick] (0,0) -- (-.5,-.5);

\draw[thick] (3,0) -- (2.5,.5);

\draw[thick] (3,0) -- (3.5,.5);
\draw[thick] (3,0) -- (3.5,-.5);

\end{tikzpicture}
\]
Similarly when the color is $-$, we have the following knot and Tait diagrams
\[
\begin{tikzpicture}[x=.35cm, y=.35cm,
    every edge/.style={
        draw,
      postaction={decorate,
                    decoration={markings}
                   }
        }
] 

\node (bbb) at (0,-.5) {$\scalebox{.5}{$-$}$}; 
\node (bbb) at (1,-.5) {$\scalebox{.5}{$+$}$}; 
\node (bbb) at (2,-.5) {$\scalebox{.5}{$-$}$}; 
\node (bbb) at (3,-.5) {$\scalebox{.5}{$+$}$}; 

\fill (0,0)  circle[radius=1.5pt];
\fill (1,0)  circle[radius=1.5pt];
\fill (2,0)  circle[radius=1.5pt];
\fill (3,0)  circle[radius=1.5pt];

\draw[thick] (0,0) to[out=90,in=90] (1,0);
\draw[thick] (1,0) to[out=90,in=90] (2,0);
\draw[thick] (2,0) to[out=-90,in=-90] (3,0);
 
\draw[thick] (0,0) -- (-.5,.5);
\draw[thick] (0,0) -- (-.5,-.5);

\draw[thick] (3,0) -- (2.5,.5);

\draw[thick] (3,0) -- (3.5,.5);
\draw[thick] (3,0) -- (3.5,-.5);

\end{tikzpicture}
\qquad 
\begin{tikzpicture}[x=.3cm, y=.3cm, every path/.style={
 thick}, 
every node/.style={transform shape, knot crossing, inner sep=1.5pt}]


\node (a1) at (0,0) {};
\node (a2) at (2,2) {};

\node (b3) at (0,2) {};
\node (b2) at (1,1) {};
\node (b1) at (2,0) {};

\draw[<-] (a1.center) .. controls (a1.4 north east) and (a2.4 south west) .. (a2.center);
\draw[<-] (b1.center) .. controls (b1.4 north west) and (b2.4 south east) .. (b2);
\draw (b2) .. controls (b2.4 north west) and (b3.4 south east) .. (b3.center);


\node (a3) at (4,0) {};
\node (a4) at (6,2) {};

\node (b6) at (4,2) {};
\node (b5) at (5,1) {};
\node (b4) at (6,0) {};

\draw[->] (a3.center) .. controls (a3.4 north east) and (a4.4 south west) .. (a4.center);
\draw (b4.center) .. controls (b4.4 north west) and (b5.4 south east) .. (b5);
\draw (b5) .. controls (b5.4 north west) and (b6.4 south east) .. (b6.center);

\draw (b6.center) .. controls (b6.4 north west) and (a2.4 north east) .. (a2.center);
\draw (a3.center) .. controls (a3.4 south west) and (b1.4 south east) .. (b1.center);


\node (aa1) at (8,-3) {};
\node (aa2) at (10,-1) {};

\node (bb3) at (8,-1) {};
\node (bb2) at (9,-2) {};
\node (bb1) at (10,-3) {};

\draw[<-] (bb1.center) .. controls (bb1.4 north west) and (bb3.4 south east) .. (bb3.center);
\draw[<-]  (aa1.center) .. controls (aa1.4 north east) and (bb2.4 south west) .. (bb2);
\draw (bb2) .. controls (bb2.4 north east) and (aa2.4 south west) .. (aa2.center);

\draw (a4.center) .. controls (a4.4 north east) and (bb3.4 north west) .. (bb3.center);
\draw (aa1.center) .. controls (aa1.4 south west) and (b4.4 south east) .. (b4.center);

\end{tikzpicture}
\]
and after smoothing the leftmost crossing, we get the following Tait diagram
\[
\begin{tikzpicture}[x=.35cm, y=.35cm,
    every edge/.style={
        draw,
      postaction={decorate,
                    decoration={markings}
                   }
        }
] 

\node (bbb) at (0,-.5) {$\scalebox{.5}{$-$}$}; 
\node (bbb) at (1,-.5) {$\scalebox{.5}{$+$}$}; 
\node (bbb) at (2,-.5) {$\scalebox{.5}{$-$}$}; 
\node (bbb) at (3,-.5) {$\scalebox{.5}{$+$}$}; 

\fill (0,0)  circle[radius=1.5pt];
\fill (1,0)  circle[radius=1.5pt];
\fill (2,0)  circle[radius=1.5pt];
\fill (3,0)  circle[radius=1.5pt];

\draw[thick] (0,0) to[out=90,in=90] (1,0);
\draw[thick] (1,0) to[out=90,in=90] (2,0);
\draw[thick] (0,0) to[out=-90,in=-90] (1,0);
\draw[thick] (1,0) to[out=-90,in=-90] (2,0);
\draw[thick] (2,0) to[out=-90,in=-90] (3,0);
 
\draw[thick] (0,0) -- (-.5,.5);
\draw[thick] (0,0) -- (-.5,-.5);

\draw[thick] (3,0) -- (2.5,.5);

\draw[thick] (3,0) -- (3.5,.5);
\draw[thick] (3,0) -- (3.5,-.5);

\end{tikzpicture}
\]
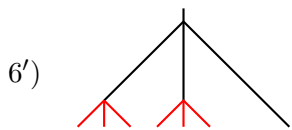
\begin{figure}

\begin{tikzpicture}[x=.35cm, y=.35cm,
    every edge/.style={
        draw,
      postaction={decorate,
                    decoration={markings}
                   }
        }
]
\node (bbb) at (-2,2) {$\scalebox{1}{$6')$}$}; 

\draw[thick] (1,1)--(4,4)--(8,0);
\draw[thick, red] (1,1)--(0,0);
\draw[thick] (4,4.5)--(4,4);
\draw[thick, red] (2,0)--(1,1);

\draw[thick, red] (1,0)--(1,1);
\draw[thick, red] (4,0)--(4,1);
\draw[thick] (4,1)--(4,4);
\draw[thick, red] (3,0)--(4,1)--(5,0);

%

\end{tikzpicture}
\caption{An alternative move of type 6.
The red part is the one to be removed.
}\label{additional6move}
\end{figure}
In both cases we get a new element $(T_+',T_-')$, where $T_+'$ has $n-4$ leaves.
\end{proof}

\section*{Acknowledgements}
We would like to thank the referees for their particularly attentive perusal of the manuscript, which resulted in many improvements in the presentation of the results of this paper.
The authors acknowledge the support by the Swiss National Science foundation through the SNF project no. 178756 (Fibred links, L-space covers and algorithmic knot theory).
 V.A. acknowledges the support  of the  Mathematisches Institut of the University of  Bern.


\end{document}